\documentclass{article}

\usepackage{amsmath,amsfonts,amsthm,amstext,amssymb,amscd,amsbsy}

\usepackage{indentfirst}
\usepackage[dvips]{color}
\usepackage{color}
\usepackage{epsfig}
\usepackage{hyperref}

\usepackage{lscape}
\usepackage{pdflscape}

\usepackage{pstricks, pst-plot}
\usepackage{pictex,dcpic}
\usepackage{tikz}
\usepackage{multicol}
\usetikzlibrary{arrows,decorations.markings}
\usetikzlibrary{matrix}
\usetikzlibrary{graphs}
\usetikzlibrary{backgrounds}

\newtheorem{theorem}{Theorem}[section]
\newtheorem{proposition}[theorem]{Proposition}
\newtheorem{lemma}[theorem]{Lemma}
\newtheorem{corollary}[theorem]{Corollary}
\newtheorem{definition}[theorem]{Definition}
\newtheorem{remark}[theorem]{Remark}
\newtheorem{example}[theorem]{Example}

\newcommand{\RR}{\mathbb R}

\newcommand{\mf}{\mathfrak}
\newcommand{\ad}{\mbox{ad}}
\newcommand{\F}{\mathbb{F}_{\Theta}}
\newcommand{\Pt}{R_{\Theta}}
\newcommand{\p}{\mf{p}}
\newcommand{\m}{\mf{m}}
\newcommand{\Rt}{R_{\mf{t}}}
\def\DynkinNodeSize{3.5mm}
\def\DynkinArrowLength{3mm}
\tikzset{
  dnode/.style={
    circle,
    inner sep=0pt,
    minimum size=\DynkinNodeSize,
    fill=white,
    draw},
  middlearrow/.style={
    decoration={markings,
      mark=at position 0.6 with
      {\draw (0:0mm) -- +(+135:\DynkinArrowLength); \draw (0:0mm) -- +(-135:\DynkinArrowLength);},
    },
    postaction={decorate}
  },
  leftrightarrow/.style={
    decoration={markings,
      mark=at position 0.999 with
      {
      \draw (0:0mm) -- +(+135:\DynkinArrowLength); \draw (0:0mm) -- +(-135:\DynkinArrowLength);
      },
      mark=at position 0.001 with
      {
      \draw (0:0mm) -- +(+45:\DynkinArrowLength); \draw (0:0mm) -- +(-45:\DynkinArrowLength);
      },
    },
    postaction={decorate}
  },
  sedge/.style={
  },
  dedge/.style={
    middlearrow,
    double distance=0.5mm,
  },
  tedge/.style={
    middlearrow,
    double distance=1.0mm+\pgflinewidth,
    postaction={draw}, 
  },
  infedge/.style={
    leftrightarrow,
    double distance=0.5mm,
  },
}
\begin{document} 
\title{$\mathcal{G}_1$ structures on flag manifolds}

\author{Luciana A. Alves \thanks{Federal University of Uberl\^andia, luciana.postigo@gmail.com}\hspace{.1cm} and Neiton Pereira da Silva \thanks{Federal University of Uberl\^andia, neitonps@gmail.com}}
\maketitle




\begin{abstract}
 
Let $\F=U/K_\Theta$ be a generalized flag manifold, where $K_\Theta$ is the centralizer of a torus in  $U$. We study $U$-invariant almost Hermitian structures on $\F$. The classification of these structures are naturally related with the system $R_t$ of t-roots associated to $\F$. We introduced the notion of connectedness by triples zero sum in a general set of linear functional and proved that t-roots are connected by triples zero sum. Using this property, the invariant $\mathcal{G}_1$ structures on $\F$ are completely classified. We also study the K\"ahler form and classified the invariant quasi K\"ahler structures on $\F$, in terms of t-roots.

\end{abstract}



\section{Introduction}

An almost Hermitian manifold is a differentiable manifold of even dimension $M$ endowed with a almost complex structure $J$ and a Riemannian metric $g(\cdot, \cdot )$ such that  $g( JX, JY ) =g( X, Y )$, for all $X, Y \in \mathfrak{X}(M)$. Let $\Omega (X,Y)=g(JX,Y)$ be the K\"ahler form and $\nabla$ the Riemannian connection. The pair $(g,J)$ is a K\" ahler structure on $M$ if $J$ is integrable and the associated K\"ahler form is closed (i.e. $d\Omega=0$) or equivalently $\nabla J=0$. According to \cite{GH}, various authors have studied certain types of almost Hermitian manifolds with the aim of generalizing geometry K\" ahler. A pre-K\"ahlerian structure has linear type if 
\begin{eqnarray}\label{pre k}
(\nabla_X J)Y+A_1 J(\nabla_{JX}J)Y +A_2 J(\nabla_X J)Y + A_3(\nabla_{JX}J)Y + A_4 (\nabla_Y J)X\nonumber \\+ A_5 J (\nabla_{JY}J)X + A_6 J(\nabla_Y J)X + A_7 (\nabla_{JY}J)X=0
\end{eqnarray}
with $A_i\in\mathbb{R}$, (see \cite{Vidal-Hervella}). For instance, the structure $(g,J)$ is quasi-K\"ahler if $(\nabla_XJ)Y-J(\nabla_{JX}J)Y=0$ for all $X,Y \in \mathfrak{X}(M)$; nearly K\"ahler if $(\nabla_X J)Y+J(\nabla_Y J)X=0$; almost K\"ahler if $d\Omega=0$ and semi-K\"ahlerian if $\delta\Omega(X)=0$, where $\delta\Omega(X)$ denotes co-derivative of the K\"ahler form, (cf. \cite{G}). It is usual to denotes these structures by QK, NK, AK and SK, respectively.

Hervella and Vidal, found in \cite{Vidal-Hervella} a new pre-K\"aherian structures, which they called by \emph{$\mathcal{G}_1$ structures}. The pair $(g,J)$ is a $\mathcal{G}_1$ structure if it satisfies $g(N(X,Y),X)=0$, where $N(X,Y)$ denotes the Nijenhuis tensor of the structure $J$. They also proved that any polynomial of the form \ref{pre k}, with $A_i\in \mathbb{R}$, gives a known pre-K\"ahlerian structure or the $\mathcal{G}_1$ structures.  

The $\mathcal{G}_1$ structures have an interested property: Let $\varphi\colon (M, g, J)\longrightarrow(M^0, g^0, J^0)$ be a conformal diffeomorphism, where $J^0X^0=(JX)^0$. Then SK, NK and AK are not conserved by $\varphi$, (see \cite{G}). On the other hand, $\mathcal{G}_1$ structures are conformally invariant, (cf. \cite{F-H}). 

Let $G$ be a complex simple Lie group with Lie algebra $\mathfrak{g}$ and $P_\Theta$ a parabolic subgroup of $G$, a generalized flag manifold (or simply flag manifold) is the homogeneous space $\F =G/P_\Theta$. If $U$ is a compact real form of $G$, then $\F=U/K_\Theta$, where $K_\Theta$ is the centralizer of a torus in  $U$. We denote by $\mathfrak{u}$ and $\mathfrak{k}_\Theta$ the Lie algebra of $U$ and $K_\Theta$ respectively.  We study $U$-invariant almost Hermitian structures on  flag manifolds $\F$. Such structures consists of a pair $(g,J)$ where $g$ is a $U$-invariant Riemannian metric and $J$ is a $U$-invariant almost complex structure on $\F$, which is compatible with the metric $g$. It is well known that for each invariant complex structure $J$ on $\F$, there exist a unique, up to homotheties, invariant metric $g$ such that the pair is $(g, J)$ is a K\"ahler structure and $g$ is an Einstein metric (see \cite{BH}; \cite{Besse} 8.95). 

  In \cite{GH}, all almost Hermitian structures were classified  into sixteen classes. In the case of flag manifolds some of these classes coincide, for example, in \cite{SM N}  San Martin-Negreiros  proved that for most of full flag manifolds the sixteen classes of Gray-Hervella reduce to three classes: K\"ahler (K), Quasi K\"ahler - QK  (or (1,2)-symplectic) and $\mathcal{G}_1$ structures (or $W_1\oplus W_3\oplus W_4$). Besides they proved that the only exception is the full flag of type $A_2$, where there exist an invariant almost Hermitian structure nearly K\"ahler which is not K\"ahler. Thus, in the case of full flag manifolds, they got a positive answer for the Wolf-Gray conjecture, which is: \emph{  Let $U/K$ be a homogeneous space of a compact Lie group $U$ which is not Hermitian symmetric and such that the isotropy $K$ has maximal rank. Then there are invariant almost Hermitian structures on $U/K$ which are nearly K\"ahler but not K\"ahler if and only if the isotropy subalgebra is the fixed point set of an autormorphism of order three}, (cf. \cite{Wolf-Gray}). This conjecture was also proved for generalized flag manifolds in \cite{SM R}. 

Concerning the three classes of invariant almost Hermitian structures on flag manifolds, the K\"ahler structures on flag manifolds were extensively studied by several authors (see for example \cite{Alek e Perol} and \cite{Besse} p.224),  class (1,2)-symplectic was completed classified and well understood in \cite{SM N}. In this paper we obtain classify completely and obtain a well understood of $\mathcal{G}_1$ structures on  flag manifolds $\F$.

An important invariant of a flag manifold $\F$ is the system of t-roots $R_t$, which is a set of linear functional defined as restriction of the root system $R$ of $\mathfrak{g}$ to real form $\mathfrak{t}$ of the center $Z(\mf{k}_{\Theta}^\mathbb{C})$ of the complexification of the subalgebra $\mf{k}_{\Theta}$. The system of t-roots $R_t$ provides a unique decomposition of tangent space $T_o\F=\m_1\oplus\cdots\oplus\m_s$ into irreducible and inequivalent $ad(\mf{k}_{\Theta})$-submodules (\cite{Sie}). Thus all $G$-invariant tensor on $\F$ can be described by means of t-roots. For example, any invariant almost complex structure $J$ on a flag manifold $\F$ is determined by a set of signs $\{\varepsilon_\delta = \pm 1, \delta\in R_t\}$, satisfying $-\varepsilon_\delta=\varepsilon_{-\delta}$, (see Proposition \ref{iacs}) and any invariant metric $g$ on $\F$ has the form  $g(\cdot, \cdot)=- \lambda_1 B(\cdot,\cdot)|_{\m_1\times\m_1}-\cdots-\lambda_s B(\cdot,\cdot)|_{\m_s\times\m_s}$, where $B(\cdot,\cdot)$ denotes the Killing form of $\mf{g}$. Thus the definition condition of any class from the sixteen classes of Gray - Hervella reduces to a algebraic expression depending on the parameters $\lambda_\delta$ and $\varepsilon_\delta$. 

We define the term connected by triple zero sum for a general set of linear functional, which for the system of t-roots $R_t$ can be state as following:\emph{ we say that two t-roots $\eta, \delta \in R_t$, $\eta\neq \pm \delta$, are connected by triples zero sum (tzs) if there exists a chain of triples $T_{i,j,k}^p=\{\xi_i,\xi_j,\xi_k\}\subset R_t$ with $\xi_i+\xi_j+\xi_k=0$ such that $\pm\eta\in T_{i,j,k}^1$, $\pm\delta\in T_{i,j,k}^n$ and $T_{i,j,k}^p\cap T_{i^{\prime},j^{\prime},k^{\prime}}^{p+1}\neq \emptyset$,  $1\leq p\leq n$.}  The key of our study on $\mathcal{G}_1$ -structures on $\F$ is the following property we prove: For any flag manifold the associated system $R_t$ of t-roots is connected by tzs, i.e., any two (non symmetric) t-roots in $R_t$ can be connected by tzs (see Theorem \ref{connected tzs t-roots}). Using this result we prove, in section 10, that the metric induced by the Killing form on $\F$ (the normal metric) is the unique invariant metric on $\F$ which is a $\mathcal{G_1}$ structure with respect to any invariant almost complex structure (see Theorem \ref{metric G1}). We think that the notion of connectedness by tzs can be very useful for the study of others invariants geometric aspect on flag manifolds.

In section 2 we give the description of flag manifold as complex and real homogeneous spaces using Lie theory. In section 3 we introduce the notion of connection by triples zero sum for roots of the system root $R$ of $\mf{g}$ and proved that in any irreducible root system, any two roots are connected by tzs. This section is a preparation to investigate the connection by tzs in t-roots. 

In section 6 we study the invariant almost complex structures (abbreviated by iacs) on $\F$. We get a description of iacs on flag manifolds in terms of t-roots, which extends the similar result in \cite{SM N} and  \cite{SM R}. As a consequence of this we obtain that every iacs $J$ on a isotropy irreducible flag manifold is integrable. We also give an example of a iacs on a flag manifold with two isotropy summand which is not integrable.

In section 9 we present a study of the K\"ahler form $\Omega$ on $\F$ associated to the structure $(g, J)$, which extends the results obtained for full flag manifolds obtained in \cite{SM N}. In particular, we obtain that classes almost K\"ahler and K\"ahler coincide on $\F$.

In the last section we show how two structures $(g, J)$ and $(g^0, J^0)$ can be equivalent by means of the action of the Weyl group of $G$ on these structures. Equivalence here means that $(g, J)$ and $(g^0, J^0)$ are associated by a bi-homorphic map $\F\longrightarrow\F$. Thus equivalent structures share the same class of invariant almost Hermitian structures. 

\section{Flag manifolds}
\label{sec:1}
In this section we set up our notation and present the standard
theory of partial (or generalized) flag manifolds associated with semisimple
Lie algebras (see for example \cite{SM N} for similar description).

 Let $\mathfrak{g}$ be a finite-dimensional semisimple complex Lie
algebra and $G$ a Lie group with Lie algebra $\mathfrak{g}$. Consider a Cartan subalgebra $\mathfrak{h}$ of $\mathfrak{g}$. We denote by $R$ the system of
roots of $(\mathfrak{g},\mathfrak{h})$. A root $\alpha\in R$ is a linear
functional on $\mathfrak{g}$. It determines uniquely an element
$H_{\alpha}\in\mathfrak{h}$ by the Riesz representation $\alpha(X)=B(
X,H_\alpha)$, $X\in\mathfrak{g}$, with respect to the Killing form
$B(\cdot,\cdot)$ of $\mf{g}$. The Lie algebra $\mathfrak{g}$ has the following
decomposition
\[
\mathfrak{g}=\mathfrak{h}\oplus\sum_{\alpha\in R}\mathfrak{g}_{\alpha}
\]
where $\mathfrak{g}_{\alpha}$ is the one-dimensional root space corresponding to $\alpha$. Besides the eigenvectors $E_{\alpha}\in\mathfrak{g}_{\alpha}$ satisfy the following equation

\begin{equation}
\left[  E_{\alpha},E_{-\alpha}\right]  =B\left(
E_{\alpha},E_{-\alpha}\right) H_{\alpha}. \label{igualdade}
\end{equation}

We fix a system $\Sigma$ of simple roots of $R$ and denote by
$R^+$ and $R^{-}$ the corresponding set of positive and
negative roots, respectively. Let $\Theta\subset\Sigma$ be a
subset, define
\begin{align}
R_{\Theta}:=\langle\Theta\rangle\cap R \quad \text{and}\quad
\Pt^\pm:=\langle\Theta\rangle\cap R^{\pm}.\nonumber
\end{align}
We denote by $R_M:=R\setminus\Pt$ the complementary set of roots. In general, $R_M$ is not a root system.

\begin{example}\label{ex1}
In the Lie algebras $A_3$, $\Sigma=\{\alpha_1,\alpha_2, \alpha_3\}$. If $\Theta=\{\alpha_2,\alpha_3\}$ then $R_M=\{\pm\alpha_1,\pm(\alpha_1+\alpha_2), \pm(\alpha_1+\alpha_2+\alpha_3)\}$ and it is not a root system.
\end{example}

Recall that $R$ is irreducible if and only if $R$ (or, equivalently, $\Sigma$) cannot be partitioned into two proper, orthogonal subsets (see \cite{H}). Equivalently, $R$ is irreducible iff the Dynkin diagram of $\Sigma$ is connected. From the classification of the connected Dynkin diagrams, if $\mf{g}$ is a simple Lie algebra then its root system $R$ is irreducible. 
\begin{lemma}\label{ponte}
Let $R$ be a irreducible root system with a simple root system $\Sigma$. Let $\Theta\subset \Sigma$ be a subset such that $\Sigma-\Theta=\Delta_1\cup\Delta_2$ is decomposed into two proper, orthogonal subsets, $\Delta_1$ and $ \Delta_2 $. Then there exists a root $\beta \in R_M$ of the form
$$
\beta=\alpha^1+\phi+\alpha^2
$$
where $\alpha^i\in \Delta_i$, $i=1,2$, and $\phi\in R_{\Theta}$.
\end{lemma}
\begin{proof}
Since the Dynkin diagram of $\Sigma$ is connected, there exists $\alpha^i\in \Delta_i$, ($i=1, 2$) which $(\alpha^1,\phi_1)\neq 0$ and $(\alpha^2,\phi_k)\neq 0$ for some $\phi_1,\phi_k\in \Theta$, where $(\alpha^i,\phi_j)=B(H_{\alpha^i},H_{\phi_j})$. Note that $(\alpha^1,\phi_1)< 0$ and $(\alpha^2,\phi_k)< 0$ then $\alpha^1+\phi_1,\alpha^2+\phi_k\in R$, because $\alpha^1,\alpha^2,\phi_1$ and $\phi_k$ are simple roots (see \cite{H}, Lemma 9.4).

From the connectedness of the Dynkin diagram of $\Theta$ we can see that there exist $\phi_2,\phi_{k-1}\in \Theta$ which satisfy
$$
(\phi_1,\phi_2)<0\quad\text{and} \quad (\alpha^1,\phi_2)=0 \quad \text{then}\quad (\alpha^1+\phi_1,\phi_2)<0 \quad \text{and}\quad \alpha^1+\phi_1+\phi_2\in R
$$
$$
(\phi_{k-1},\phi_k)<0\quad\text{and} \quad (\phi_{k-1},\alpha^2)=0 \quad \text{then}\quad (\phi_{k-1}, \phi_k+\alpha^2)<0 \quad \text{and}\quad \phi_{k-1}+ \phi_k+\alpha^2\in R.
$$
Continuing with this process we obtain $\phi_3,\dots,\phi_{k-2}\in \Theta$ such that 
$$
(\alpha^1+\phi_1+\phi_2)+(\phi_3+\dots+\phi_{k-2})+(\phi_{k-1}+\phi_k+\alpha^2)\in R.
$$
Note that $\phi=\phi_1+\dots+\phi_k\in R_{\Theta}$ and $\alpha^1+\phi+\alpha^2\in R\setminus R_{\Theta}=R_M$.
\end{proof}

Next we discuss the flag manifolds as homogeneous spaces. Note that
\[
\p_{\Theta}:=\mf{h}\oplus\sum_{\alpha\in R^+}\mf{g}_{\alpha}\oplus\sum_{\alpha\in R_{\Theta}^-}\mf{g}_{\alpha}
\]
is a parabolic subalgebra, since it contains the Borel subalgebra
$\mf{b}^+=\mf{h}\oplus\sum\limits_{\alpha\in R^+}\mf{g}_{\alpha}$.

The partial flag manifold determined by the choice $\Theta\subset R$ is the
homogeneous space $\F=G/P_{\Theta}$, where $P_{\Theta}$ is the normalizer of
$\mf{p}_{\Theta}$ in $G$. In the special case $\Theta=\emptyset$, we obtain the \emph{full} (or maximal) flag manifold $\mathbb{F}=G/B$ associated with $R$, where
$B$ is the normalizer of the Borel subalgebra $\mf{b}^+=\mf{h}\oplus\sum\limits_{\alpha\in R^+}\mf{g}_{\alpha}$ in $G$.

Now we will see the construction of any flag manifold as the
quotient $U/K_{\Theta}$ of a semisimple compact Lie group
$U\subset G$ modulo the centralizer $K_{\Theta}$ of a torus in
$U$. We fix once and for all a Weyl base of $\mf{g}$ which amounts
to giving $X_\alpha\in\mf{g_{\alpha}}$, $H_{\alpha}\in\mf{h}$ with $\alpha\in R$, with the standard
properties:
\begin{equation}\label{base weyl}
\begin{tabular}
[c]{lll}
$B( X_{\alpha},X_{\beta}) =\left\{
\begin{array}
[c]{cc}%
1, & \alpha+\beta=0,\\
0, & \text{otherwise};
\end{array}
\right. $ &\hspace{-.5cm} & $\left[  X_{\alpha},X_{\beta}\right]  =\left\{
\begin{array}
[c]{cc}%
H_{\alpha}\in\mathfrak{h}, & \alpha+\beta=0,\\
n_{\alpha,\beta}X_{\alpha+\beta}, & \alpha+\beta\in R,\\
0, & \text{otherwise.}%
\end{array}
\right.  $
\end{tabular}
\end{equation}
The real constants $n_{\alpha,\beta}$ are non-zero if and only if
$\alpha+\beta\in R$. Besides that it satisfies
$$
\left\{
\begin{array}
[c]{cc}%
 n_{\alpha,\beta}=-n_{-\alpha,-\beta}=-n_{\beta,\alpha}& \\
\hspace*{-1cm} n_{\alpha,\beta}=n_{\beta,\gamma}=n_{\gamma,\alpha},&\mbox{if}\quad\alpha+\beta+\gamma=0.
\end{array}
\right.
$$

We consider the following two-dimensional real spaces
$\mf{u}_{\alpha}=\text{span}_\mathbb{R}\{A_{\alpha},iS_{\alpha}\}$, 
where $\hspace{0.3cm}A_{\alpha}=X_{\alpha}-X_{-\alpha}$
and $\hspace{0.3cm}S_{\alpha}=X_{\alpha}+X_{-\alpha}$, with
$\alpha\in R^{+}$.
Then the real Lie algebra
$\mf{u}=i\mf{h}_\mathbb{R}\oplus\sum\mf{u}_{\alpha}$, with $\alpha\in R^{+},$
is a compact real form of $\mf{g}$, where $\mf{h}_\mathbb{R}$ denotes the real space vector spanned by $\{H_\alpha; \alpha\in R\}$.

Let $U=\exp\mf{u}$ be the compact real form of $G$ corresponding
to $\mf{u}$. By the restriction of the action of $G$ on $\F$, we
can see that $U$ acts transitively on $\F$, then $\F=U/K_{\Theta}$,
where $K_{\Theta}=P_{\Theta}\cap U$. The Lie algebra $\mf{k}_{\Theta}$ of $K_\Theta$ is the set of fixed points of the conjugation $\tau\colon X_{\alpha}\mapsto-X_{-\alpha}$ of $\mf{g}$ restricted to $\p_{\Theta}$
$$
\mf{k}_{\Theta}=\mf{u}\cap\p_{\Theta}=i\mf{h}_\mathbb{R}\oplus\sum_{\alpha\in\Pt^+}\mf{u}_{\alpha}.
$$

The tangent space of $\F=U/K_{\Theta}$ at the origin $o=eK_{\Theta}$ can be identified with the orthogonal complement (with respect to the Killing form) of $\mf{k}_{\Theta}$ in $\mf{u}$
$$
T_o\F=\m=\sum\limits_{\alpha\in R_{M}^{+}}\mf{u}_{\alpha}
$$
with $R_{M}^{+}=R_{M}\cap R^+$. Thus we have $\mf{u}=\mf{k}_{\Theta}\oplus\mf{m}$.

It is known that there is a one-to-one correspondence between flag manifolds $\F$ of a compact semisimple Lie group (up to isomorphism) and painted Dynkin diagrams (up to equivalence) (cf. \cite{Alek e Perol} or \cite{Arv art2}). Next we present a briefly description of this correspondence. 

Two flag manifolds $\mathbb{F}=G/K$ and
$\mathbb{F}^{\prime}=G/K^{\prime}$ are \textit{equivalent} if
there exist an automorphism $\Phi\in Aut\left( G\right)$ such that
$\Phi\left(  K\right) =K^{\prime}$. This automorphism $\Phi$ induce a diffeomorphism $\ \widetilde{\Phi}\colon \mathbb{F}\longrightarrow
\mathbb{F}^{\prime}$ defined by $\ \widetilde{\Phi}\left(
gK\right) =\Phi\left( g\right) K^{\prime}$.

Let $\F$ be a flag manifold determined by $\Theta \subset \Sigma$.
Let $\Gamma=\Gamma(\Sigma)$ be the Dynkin diagram of the root system $R$. In $\Gamma$ we obtain the painted Dynkin diagram of $\F$ by painting the nodes $\Sigma - \Theta$ in black. Thus the simple roots $\Theta$ correspond to the subdiagram of white nodes.
\newpage
\begin{example}
The flag manifold $\F=\frac{SU(5)}{S(U(3)\times
U(2))}$ corresponds to the painted diagram
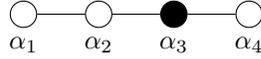
\begin{center}

\begin{multicols}{2}

\begin{tikzpicture}
    \draw (-1,0) node[anchor=east]  {};
    
    \node[dnode,label=below:$\alpha_1$] (1) at (0,0) {};
    \node[dnode,label=below:$\alpha_2$] (2) at (1,0) {};
    \node[dnode,label=below:$\alpha_3$, fill=black] (3) at (2,0) {};
    \node[dnode,label=below:$\alpha_4$] (4) at (3,0) {};

    \path (1) edge[sedge] (2)
          (2) edge[sedge] (3)
          (3) edge[sedge] (4)
          ;
\end{tikzpicture}
\end{multicols}
\end{center}

\end{example}
Note that maximal flag manifolds correspond to Dynkin diagrams with all roots painted in black.

Conversely, let $\Gamma$ be the Dynkin diagram of a semisimple Lie algebra $\mathfrak{g}$. Suppose that some nodes in $\Gamma$ are painted in black. Then 
\[
\mf{k}_{\Theta}= \mathfrak{u}(1)\oplus\cdots\oplus\mathfrak{u}(1)\oplus\mf{k}_{\Theta}^{\prime}
\]
where the semisimple part of $\mf{k}_{\Theta}$, denoted by $\mf{k}_{\Theta}^{\prime}$, is yielded by the set of white nodes together with the connected line between them, and each black nodes in $\Gamma$ yields a $\mathfrak{u}(1)$ component.

\begin{example}

The painted Dynkin diagram of the Lie algebra $\mathfrak{e_8}$ below
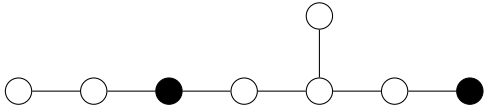
\begin{center}

\begin{multicols}{2}

\begin{tikzpicture}
    \draw (-1,0) node[anchor=east]  {};
    
    \node[dnode,] (1) at (0,0) {};
    \node[dnode,] (2) at (1,0) {};
    \node[dnode, fill=black] (3) at (2,0) {};
    \node[dnode,] (4) at (3,0) {};
    \node[dnode,] (5) at (4,0) {};
    \node[dnode,] (6) at (4,1) {};
    \node[dnode,] (7) at (5,0) {};
    \node[dnode,fill=black] (8) at (6,0) {};
    \path (1) edge[sedge] (2)
          (2) edge[sedge] (3)
          (3) edge[sedge] (4)
          (4) edge[sedge] (5)
          (5) edge[sedge] (6)
          (5) edge[sedge] (7)
          (7) edge[sedge] (8)
          ;
\end{tikzpicture}
\end{multicols}
\end{center}
corresponds to the flag manifold $\dfrac{E_8}{U(1)^2\times SU(3)\times SO(8)}$.
\end{example}

\section{Connectedness by triples zero sum}

In this section we define connectedness by triples zero sum for a general set of linear functional. In the case of root system we prove that this property is equivalenty to the connectedness of the Dynkin diagram.  We will show that the study of this notion on certain sets of linear functional is the key for a complete classification of $\mathcal{G}_1$ structures on $\F$. 

\begin{definition}
Let $\Gamma$ be a non empty set of linear functional. Two linear functinal $\alpha, \beta$ in $\Gamma$ such that $\alpha\neq \pm \beta$ are connected by triples zero sum (in $\Gamma$), if there exists a chain of triples $T_{i,j,k}^p=\{\gamma_i,\gamma_j,\gamma_k\}\subset \Gamma$ with $\gamma_i+\gamma_j+\gamma_k=0$ such that $\pm\alpha\in T_{i,j,k}^1$, $\pm\beta\in T_{i,j,k}^n$ and $T_{i,j,k}^p\cap T_{i^{\prime},j^{\prime},k^{\prime}}^{p+1}\neq \emptyset$ with $1\leq p \leq n-1$. We will say that $\Gamma$ is connected by triples zero sum (or connected by tzs) if any pair $\alpha, \beta \in \Gamma$ are connected by triples zero sum.
\end{definition}

If we adopt that each linear functional $\alpha \in \Gamma$ is connected by tzs with itself, then connectedness by tzs becomes a equivalence relation in set of linear functional $\Gamma$. 
\begin{example}
Consider the Lie algebras of type $A_5$. Its Dynkin diagram is given by
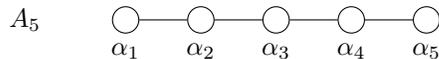
\begin{multicols}{2}
\begin{tikzpicture}
    \draw (-1,0) node[anchor=east]  {$A_5$};
    
    \node[dnode,label=below:$\alpha_1$] (1) at (0,0) {};
    \node[dnode,label=below:$\alpha_2$] (2) at (1,0) {};
    \node[dnode,label=below:$\alpha_3$] (3) at (2,0) {};
    \node[dnode,label=below:$\alpha_4$] (4) at (3,0) {};
    \node[dnode,label=below:$\alpha_5$] (5) at (4,0) {};

    \path (1) edge[sedge] (2)
          (2) edge[sedge] (3)
          (3) edge[sedge] (4)
          (4) edge[sedge] (5);
\end{tikzpicture}
\end{multicols}
If $\Gamma$ is the set of roots generated by $\alpha_1, \alpha_3$ and $\alpha_4$ then $\Gamma$ is not connected by tzs.
\end{example}

\begin{lemma}\label{tzs for root system}
All irreducible root system is connected by tzs.
\end{lemma}
\begin{proof}
Let $R$ be a irreducible root system with basis $\Sigma=\{\alpha_1,\dots,\alpha_n\}$ and consider the correspondent set $R^+$ of positive roots. Since $R$ is irreducible, $(\alpha_1,\alpha_j)\neq 0$ for some $2\leq j\leq n$. For simplicity, let $(\alpha_1,\alpha_2)\neq 0$. The roots $\alpha_1,\alpha_2$ are simple then $(\alpha_1,\alpha_2)< 0$ and $\alpha_1+\alpha_2\in R$ ( see \cite{H}, Lemma 9.4). Then $\alpha_1$ and $\alpha_2$ are connected by the triple zero sum $\{\alpha_1,\alpha_2,-(\alpha_1+\alpha_2)\}$. Applying the same idea to $\alpha_2$, we conclude that $\alpha_2$ and $\alpha_3$ are connected by the tzs $\{\alpha_2,\alpha_3,-(\alpha_2+\alpha_3)\}$. (If $(\alpha_2,\alpha_j)= 0$ for $3\leq j\leq n$ then $(\alpha_1,\alpha_j)\neq 0$ for some $3\leq j\leq n$.) Thus $\alpha_1$ and $\alpha_3$ are connected by the triple zero sum. Continuing with this process we obtain that any pair $\alpha_i, \alpha_j$ are connected by the tzs. 

Now let $\beta\in R$ not simple. It is easy to see that $\beta\pm\alpha_i\in R$ for some simple root $\alpha_i$. Then $\beta$ and $\alpha_i$ are connected by tzs. Thus every root in $R$ is connected by triple zero sum to some simple root and any pair of simple roots are connected by tzs, so any pair of roots $\beta, \gamma \in R$ are connected by tzs.
\end{proof}

If the Dynkin diagram of $R$ is not connected then $R$ cannot be connected by tzs. Since a root system $R$ is irreducible if and only if its Dynkin diagram is connected, we obtain: 

\begin{theorem}
The Dynkin diagram of $R$ is connected iff $R$ is connected by tzs. In particular, the Dynkin diagram of semisimple Lie algebras are connected by tzs.
\end{theorem}


\section{t-roots}
In order, to describe invariant tensors and to set the notation, we present the fundamental theory on the isotropy representation for flag manifolds (see, for example, \cite{Alek e Perol} or \cite{Sie}).

It is known that $\F$ is a reductive homogeneous space, this means that the adjoint representation of $\mf{k}_{\Theta}$ (and $K_{\Theta}$) leaves $\mf{m}$ invariant, i.e. $\ad(\mf{k}_{\Theta})\mf{m}\subset\mf{m}$. Thus we can decompose $\mf{m}$ into a sum of irreducible $\ad (\mf{k}_{\Theta})$ submodules $\mf{m}_i$ of the module $\mf{m}$:
\[
\m= \m_1\oplus\cdots\oplus\m_s.
\]

Now we will see how to obtain each irreducible $\ad (\mf{k}_{\Theta})$ submodules $\m_i$. By complexifying the Lie algebra of $K_{\Theta}$ we obtain
$$
\mf{k}_{\Theta}^{\mathbb{C}}=\mf{h}\oplus\sum_{\alpha\in\Pt}\mf{g}_{\alpha}.
$$
The adjoint representation $\ad(\mf{k}_{\Theta}^{\mathbb{C}})$ of $\mf{k}_{\Theta}^{\mathbb{C}}$ leaves the complex tangent space $\m^{\mathbb{C}}$ invariant.
Let
$$
\mf{t}:=Z(\mf{k}_{\Theta}^\mathbb{C})\cap i\mf{h}_\mathbb{R}
$$
the real form of the center $Z(\mf{k}_{\Theta}^\mathbb{C})$. It is easy to see that $\mathfrak{t}$ is a subalgebra of $i\mf{h}_\mathbb{R}$ orthogonal (with respect to the Killing form on $i\mf{h}_\mathbb{R}$) to $H_{\alpha}$, for all $\alpha$ in $\Pt$, i.e., 
$$
\mf{t}=\{H\in i\mf{h}_{\RR}:\alpha(H)=0,\,\text{for all} \,\alpha \in \Pt\}.
$$

Let $i\mf{h}_\RR^{\ast}$ and $\mf{t}^{\ast}$ be the dual vector
space of $i\mf{h}_\RR$ and $\mf{t}$, respectively, and consider the
map $k\colon i\mf{h}_\RR^{\ast}\longrightarrow\mf{t}^{\ast}$ given
by $k(\alpha)=\alpha|_{\mf{t}}$. The linear functional of 
$R_t:=k(R_{M})$ are called \emph{t-roots}. Denote by $R_{\mf{t}}^+=k(R_M^+)$ the set of positive t-roots. Note that the map $k$ is not a 1-1 correspondence in general.

A basis of real space $\mf{t}^{\ast}$ can be obtained as following: we fix a basis $\Sigma=\{\alpha_1, \dots \alpha_k,\beta_1, \dots, \beta_r \}$ of $R$ associated to $\F$ where $\Theta=\{\alpha_1, \dots \alpha_k\}$ is a basis of $R_\Theta$ and $\Sigma_M=\Sigma\setminus\Theta=\{\beta_1, \dots, \beta_r \}$ . Let $\bar{\beta_1},\dots, \bar{\beta_r}$ be the fundamental weights corresponding to the simple roots of $\Sigma_M$, defined by
$$
\left< \bar{\beta_i},\beta_j\right>\colon =\dfrac{2(\bar{\beta_i},\beta_j)}{(\beta_j, \beta_j)}=\delta_{ij},\quad \left<\bar{\beta_i}, \alpha_j\right>=0,
$$
where $(\alpha, \beta)=B(H_\alpha, H_\beta)$. Then $\{\bar{\beta_1},\dots, \bar{\beta_r}\}$ is a basis of $\mf{t}^{\ast}$. Indeed, let $\lambda$ in $\mf{t}^{\ast}$ be a t-root, since the Killing form $(\cdot,\cdot)$ on $i\mf{h}_\mathbb{R}$ is non degenerated and negative definite, each t-root $\lambda$ determines $H_\lambda\in\mf{t}\subset i\mf{h}_\mathbb{R}$, such that $\beta(H_\lambda)=(\lambda,\beta)$, for each root $\beta$ in $R$, by the Riez representation. Then $0=\alpha_i(H_{\lambda})=( \lambda,\alpha_i)$, for $i=1,\dots,k$.
Now let $m_j=\left< \lambda,\beta_j\right>$ then $0=\left< \lambda-\sum m_j\bar{\beta_j},\gamma\right>$, for each simple root $\gamma \in \Sigma$, which implies that $(\lambda-\sum m_j\bar{\beta_j},\gamma)=0$ as well,  or that $\lambda=\sum m_j\bar{\beta_j}$.

According to \cite{Sie}, there exists a 1-1 correspondence between positive t-roots and irreducible submodules of the adjoint representation of $\mf{k}_{\Theta}$. This correspondence is given by
\begin{equation*}
\xi\longleftrightarrow\mf{m}_{\xi}=\sum_{k(\alpha)=\xi}\mf{u}_{\alpha}
\end{equation*}
with $\xi\in R_{\mf{t}}^+$. Besides these submodules are inequivalents. Hence the tangent space can be decomposed as follows
\begin{equation*}\label{decomp}
\m=\m_{\xi_1}\oplus\cdots\oplus\m_{\xi_s}
\end{equation*}
where $R_t^+=\lbrace \xi_1, \ldots, \xi_s\rbrace$.

Following the idea of section 3, it is natural to ask if $R_t$ is connected by tzs. It is easy to see that if $R_M$ is connected by tzs then $R_t$ is connected by tzs. But the reciprocal is not true, as we can see in example \ref{ex1}, where $\Rt=\{k(\alpha_1)\}$ and $R_M=\{\pm\alpha_1,\pm(\alpha_1+\alpha_2), \pm(\alpha_1+\alpha_2+\alpha_3)\}$. The next result shows that $R_t$ inherits the connectivity property by tzs from the root system $R$.
%
%
%
%
%
%
%

\begin{theorem} \label{connected tzs t-roots}
The set of t-roots $R_t$ is connected by tzs.
\end{theorem}
\begin{proof}
Consider the Dynkin diagram of $\Sigma$. Let $N$ be the amount of connected components of the Dynkin diagram of $\Sigma-\Theta$. We will prove by mathematical induction on $N$. If $N=1$ then the diagram of $\Sigma-\Theta$ is connected, following the idea of Lemma \ref{tzs for root system} proof, we can see that  if $\alpha,\beta\in \Sigma-\Theta$ then $\alpha$ and $\beta$ are connected by tzs. Besides $\gamma\in R_M$ is not a simple root then $\gamma$ is connected by tzs to some $\alpha\in\Sigma-\Theta$. Thus any pair of roots in $R_M$ is connected by tzs. Then $R_t$ is connected by tzs.

Now suppose that $N=k+1$ then $\Sigma-\Theta=\Delta_1\cup\cdots\cup\Delta_k\cup\Delta_{k+1}$ where each connected component of the Dynkin diagram of $\Sigma-\Theta$ corresponds to the Dynkin diagram $\Delta_i$, for some $i=1, \dots,k+1$. From Lemma \ref{ponte}, there exist a complementary root $\beta\in R_M$ such that $\beta=\alpha^1+\phi+\alpha^2$, with $\alpha^1\in\Delta_k$, $\alpha^2\in\Delta_{k+1}$ and $\phi\in R_{\Theta}$. Then the correspondent t-root is $k(\beta)=k(\alpha^1)+k(\alpha^2)$ and 
$\left\langle  \bar{\Delta}_k \right\rangle $  and $\left\langle  \bar{\Delta}_{k+1} \right\rangle $ are connected by tzs, where $\left\langle  \bar{\Delta}_i \right\rangle $ denotes the t-roots generated by $k(\Delta_i)$. By induction hypothesis, 
$$
\{\left\langle  \bar{\Delta}_{1} \right\rangle\oplus\cdots\oplus\left\langle  \bar{\Delta}_{k} \right\rangle\}\cap\Rt 
$$
is connected by tzs. So 
$$
\Rt=\{\left\langle  \bar{\Delta}_{1} \right\rangle\oplus\cdots\oplus\left\langle  \bar{\Delta}_{k} \right\rangle\oplus\left\langle \bar{\Delta}_{k+1}\right\rangle \}\cap\Rt 
$$
is connected by tzs. 
\end{proof}

\section{Invariant Almost Hermitian Structures on $\F$}

An \emph{almost complex structure} on $\F=U/K_{\Theta}$ is a tensor field of type $(1,1)$ that corresponds each $x\in \F$ to a linear endomorphism  $J_{x}\colon T_{x}\F\rightarrow T_{x}\F$ which satisfies $J_{x}^{2}=-Id$. The almost complex structure on $\F$ is invariant (or $U$-invariant) if 
\[
d \mathit{u}_{x}\circ J_{x}=J_{\mathit{u}x}\circ d\mathit{u}_{x}
\]
for all $\mathit{u}\in U$. An invariant almost complex structure (\textit{iacs} from now) is determined by a linear endomorphism $J\colon\mathfrak{m\longrightarrow}\mathfrak{m}$, which satisfies $J^{2}=-Id$ and commutes with the adjoint action of $K_\Theta$ on $\mathfrak{u}$, that is,
\[
Ad(k) J=J Ad(k), \quad \text{for all},\quad  k\in K_\Theta,
\]
or, equivalently
\[
ad(L)J=Jad(L), \quad \text{ for all}, \quad L\in \mathfrak{k}_\Theta
\]
(cf. \cite{Kob}).

As it is common in the literature, we will use the same letter $J$ to denote its extension to the complexi-\\fication $\mathfrak{m}^{\mathbb{C}}$. Since $J^{2}=-Id$, its eigenvalues are $i$ and $-i$ and the correspondents eigenspaces are denoted by
\[
T^{(1,0)}_{o}\F=\left\{X\in T_{o}\mathbb{F}^{\mathbb{C}}:
JX=iX\right\}
\quad \text{and} \quad
T^{(0,1)}_{o}\F=\left\{X\in T_{o}\mathbb{F}^{\mathbb{C}}:
JX=-iX\right\}.
\]

Thus
\[
\mathfrak{m}^{\mathbb{C}}=T_{o}\F^{\mathbb{C}}=T^{(1,0)}_{o}\F\oplus
T^{(0,1)}_{o}\F.
\]
The eigenvectors with eigenvalue $+i$ (resp. $-i$) are called  \textit{of type (1,0)} (resp. \textit{of type (0,1)} ). 

The next result is known for maximal flag manifolds (see \cite{SM N}). A similar result for generalized flag manifolds can be found in \cite{SM R}. Our work here is to classify iacs in terms of t-roots. 

\begin{proposition}\label{iacs}
Let $\F$ be a generalized flag manifold and $R_t$ the correspondent set of t-roots. Then any iacs $J$ on $\F$ is determined by a set of signals $\{\varepsilon_{\delta},\delta\in R_t\}$, where $\varepsilon_{\delta}=\pm 1$, satisfying $\varepsilon_{\delta}=-\varepsilon_{-\delta}$ for all
$\delta\in R_t$. In particular, there exist $2^{|R_t^{+}|}$ iacs on $\F$ and each iacs $J$ is determined by exactly $|R_t^{+}|$ signals. 
\end{proposition}

\begin{proof}
Let $J$ be an iacs on $\F$ and consider its complexification on $\mathfrak{m}^{\mathbb{C}}$. By the invariance of $J$, if $H\in\mathfrak{h}^\mathbb{C}$ it follows $ad(H)JX_{\alpha}=\alpha(H)JX_{\alpha}$,
then $J(\mathfrak{g}_{\alpha}^{\mathbb{C}})=\mathfrak{g}_{\alpha}^{\mathbb{C}}$ with $\alpha\in R$. Since the eigenvalues of $J$ are $\pm i$ and the eigenvectors in $\mathfrak{m}^{\mathbb{C}}$ are $X_{\alpha}$,
$\alpha\in R$, we can write
$JX_{\alpha}=i\varepsilon_{\alpha}X_{\alpha}$, with
$\varepsilon_{\alpha}=\pm 1$.

Computing $J$ on the basis $\{A_\alpha, iS_\alpha\}$ of the real vector space $\mathfrak{m}$ and using that
$X_{-\alpha}=\frac{1}{2}\left(-i(iS_{\alpha})-A_{\alpha}\right)$ we obtain

\[
i\varepsilon_{-\alpha}\frac{1}{2}\left(-i(iS_{\alpha})-A_{\alpha}\right)=i\varepsilon_{-\alpha}X_{-\alpha}
=
\frac{1}{2}\left(-iJ(iS_{\alpha})-JA_{\alpha}\right).
\]
Then
\[
J(iS_{\alpha})=\varepsilon_{-\alpha}A_{\alpha} \quad\text{and}\quad JA_{\alpha}=-\varepsilon_{-\alpha}(iS_{\alpha}),
\]
thus, using $A_{\alpha}=-A_{-\alpha}$ and $iS_{\alpha}=iS_{-\alpha}$ it follows
\[
-\varepsilon_{-\alpha}(iS_{\alpha})=JA_{\alpha}=-JA_{-\alpha}=\varepsilon_{\alpha}(iS_{-\alpha})=\varepsilon_{\alpha}(iS_{\alpha})
\]
so $-\varepsilon_{-\alpha}=\varepsilon_{\alpha}$ for each $\alpha$ in $R$.

Now consider the decomposition of the complex tangent space
\begin{equation*}
\mathfrak{m}^{\mathbb{C}}=\mathfrak{m}^{\mathbb{C}}_{1}\oplus\mathfrak{m}^{\mathbb{C}}_{2}\oplus\cdots\oplus\mathfrak{m}^{\mathbb{C}}
_{2s}
\end{equation*}
in irreducible and inequivalent $\mathfrak{k}_\Theta^{\mathbb{C}}$-submodules determined by the t-roots $\delta_{1},-\delta_{1},\dots,-\delta_{s},\delta_{s}$, where
\[
\mathfrak{m}^{\mathbb{C}}_{j}=\mathfrak{m}^{\mathbb{C}}_{\delta_{j}}=\sum_{\alpha\in R,
k(\alpha)=\delta_{j}}\mathbb{C}X_{\alpha}.
\]
Since $J$ comutes with $ad(L)$, for all $L\in\mathfrak{k}_\Theta^{\mathbb{C}}$, then 
we see that
$J(\mathfrak{m}^{\mathbb{C}}_{\delta_{j}})=\mathfrak{m}^{\mathbb{C}}_{\delta_{j}}$,
for all $\delta_{j}\in R_t$. Then by Schur Lemma
\[
J =
i\varepsilon_{1}Id|_{\mathfrak{m}^{\mathbb{C}}_{1}}\oplus\cdots\oplus
i\varepsilon_{s}Id|_{\mathfrak{m}^{\mathbb{C}}_{s}}
\]
this means that, if $\alpha,\beta\in R$ are such that
$k(\alpha)=k(\beta)=\delta\in R_t$ then
$\varepsilon_{\alpha}=\varepsilon_{\beta}=\varepsilon_{\delta}$.

Finally, given a t-root $\delta$ we obtain
\[
\varepsilon_{\delta}=\varepsilon_{\alpha}=-\varepsilon_{-\alpha}=-\varepsilon_{-\delta}
\]
where $\alpha$ is any root in $R$ such that $k(\alpha)=\delta$.

\end{proof}

An iacs $J$ on $\F$ is called an \textit{invariant complex structure} (or \textit{integrable complex struture}) if its 
Nijenhuis tensor, defined by
\[
-\frac{1}{2}N(X,Y)=-\left[JX,JY\right]+\left[X,Y\right]+J\left[X,JY\right]+J\left[JX,Y\right], \quad X,Y\in\mathfrak{m}.
\]
is zero, (cf. \cite{Kob}). In the next result we obtain the necessary and suficient condition of integrability of $J$ in terms of t-roots.
\begin{proposition}\label{J complex}
An iacs $J=\{\varepsilon_{\delta},\delta\in R_t\}$ is a complex structure if and only if $\varepsilon_{\delta}\varepsilon_{\eta}+1-\varepsilon_{\delta+\eta}(\varepsilon_{\delta}+\varepsilon_{\eta})=0$, whenever $\delta+\eta\in R_t$.
\end{proposition}

\begin{proof}
An easy computation of Nijenhuis tensor on the Weyl base shows that
\begin{equation}\label{nij}
-\dfrac{1}{2}N(X_{\alpha},X_{\beta})  =\left\{
\begin{array}
[c]{cc}%
0, &\quad \text{if}\quad \alpha+\beta\notin R,\\
n_{\alpha,\beta}\{\varepsilon_{\alpha}\varepsilon_{\beta}+1-\varepsilon_{\alpha+\beta}(\varepsilon_{\alpha}+\varepsilon_{\beta})\}X_{\alpha+\beta}, &\quad \text{if}\quad \alpha+\beta \in R.
\end{array}
\right. 
\end{equation}
Now if $k(\alpha)=\delta\in R_t$ and $k(\beta)=\eta\in R_t$ then $\varepsilon_{\alpha}=\varepsilon_{\delta}$, $\varepsilon_{\beta}=\varepsilon_{\eta}$ and $\varepsilon_{\alpha+\beta}=\varepsilon_{\delta+\eta}$, if $\alpha+\beta\in R_M$.
Note that if $\alpha, \beta \in R_M$ are such that $\alpha+\beta\in R_\Theta$ then $N(X_{\alpha},X_{\beta})=0$.
\end{proof}

\begin{corollary}\label{J iso irr}
Every iacs $J$ on a isotropy irreducible flag manifold $\F$, i.e. $|R_t^+|=1$, is an invariant complex structure.
\end{corollary}

The next example shows that if $|R_t^+|\geq 2$ it is easy to construct a iacs $J$ not complex.

\begin{example} Let $\F$ be a flag manifold with two isotropy summands, i.e. the real tangent space decomposes into two irreducible submodules $\mathfrak{m}=\mathfrak{m_1}\oplus\mathfrak{m_2}$. In this case  $R_t=\{\pm\delta, \pm 2\delta\}$, where $\delta\in t^*\setminus 0$, (cf. \cite{art 1}). Using the Proposition \ref{J complex} it is easy to see that the iacs given by $J: \varepsilon_{2\delta}=1,\varepsilon_{\delta}=-1$ is not complex.
\end{example}

In the other hand, there is natural manner to obtain invariant complex structures on $\F$, as below.
Consider the space $\mathfrak{t}$ with the collection of hyperplanes
$\Gamma=\left\{H\in\mathfrak{t}:\delta(H)=0, \delta\in  R_t\right\}$. The complementary set $ \mathfrak{t}\setminus\Gamma$ of the union of these hyperplanes is 
open and dense in $\mathfrak{t}$. Its connected components in $\mathfrak{t}$ are called a chamber in $\mathfrak{t}$ or a t-chamber. Each t-chamber is a cone in $\mathfrak{t}$ which is given by the inequalities $V=\{\delta_1>0, \dots,\delta_m>0\}$, where $\delta_i$ are t-roots corresponding to its faces. 

A subset $\Sigma_t=\left\{\delta_{1},\ldots,\delta_{n}\right\}$ of t-roots is called a \emph{basis of system $R_t$} if all vectors of $R_t$  have integer coordinates of same sign $\left(\geq 0\text{ or }\leq 0\right)$ with respect to $\Sigma_t$. Note that the basis given in Section 4 is a basis of system $R_t$. 

From \cite{Alek e Perol} and \cite{BH} it is known that there exists natural one-to-one correspondence between:
\begin{enumerate}

\item Parabolic subalgebras $\mathfrak{p}_\Theta$ in $\mathfrak{g}$ with reductive part $\mf{h}\oplus\sum_{\alpha\in\Pt}\mf{g}_{\alpha}$;

\item Basis $\Sigma$ of the root system $R$ which contain a fixed basis $\Theta$ of the root system $R_{\Theta}$;
\item Basis of system $R_t$;
\item t-chambers;
\item A choice of positive roots $R^+_M$ in $R_M$ satisfying:
\begin{enumerate}
    \item $R=\Theta\cup R^+_M\cup R^-_M$ (disjoint union), where $R^-_M=\{-\alpha; \alpha\in R_M^+\}$;
    \item if $\alpha\in R_{\Theta}\cup R^+_M $, $\beta \in R^+_M$ with $\alpha+\beta\in R$, then $\alpha+\beta\in R^+_M$ (cf. \cite{Arv art2}).
\end{enumerate}
\item Invariant complex structures on $\F$, which are determined by 
\begin{equation}\label{ics}
    J X_{\pm\alpha}=\pm i X_{\pm\alpha}, \quad \text{with}\quad \alpha\in  R^+_M.
\end{equation}

\item Reflections $\omega$ of the Weyl group (of the root system $R$) which satisfy 
$\omega\Theta\subset\Sigma$.
\end{enumerate}

Now we describe the invariant metrics on flag manifolds.
A Riemannian invariant metric on $\F$ is completely determined by a real inner product $g\left(\cdot,\cdot\right)$ on $\mathfrak{m}=T_{o}\F$ which is invariant by the adjoint action of $\mf{k}_{\Theta}$. Besides that any real inner product $\ad(\mf{k}_{\Theta})$-invariant on $\mf{m}$ has the form

\begin{equation}\label{inner product}
g\left(\cdot,\cdot\right)=-\lambda_1 B\left(\cdot,\cdot\right)|_{\mathfrak{m}_{1}\times\mathfrak{m}_{1}}-\cdots
-\lambda_s B\left(\cdot,\cdot\right)|_{\mathfrak{m}_{s}\times\mathfrak{m}_{s}}
\end{equation}\\
where ${\mathfrak{m}}_{i}=\mf{m}_{\xi_i}$ and $\lambda_i=\lambda_{\xi_i}>0$ with $\xi_{i}\in R^{+}_{\mathfrak{t}}$, for $i=1,\ldots,s$. So any invariant Riemannian metric on $\F$ is determined by $|R_{\mf{t}}^+|$ positive parameters. We will call an inner product defined by (\ref{inner product}) as an invariant metric on $\F$.

An invariant metric 
 $g(\cdot, \cdot)$ is called {\it normal} if there exist a bi-invariant metric on $\mathfrak{g}$ such that the restriction to $\m=\mathfrak{k}^{\perp}$ is $g(\cdot, \cdot)$. In other words, $g(\cdot, \cdot)$ is a normal metric is 
 $$
 g(\cdot, \cdot)=-\lambda B\left(\cdot,\cdot\right)|_{\mathfrak{m}}\quad \text{or} \quad \lambda_1=\cdots=\lambda_s=\lambda.
 $$

Let $\F$ be a flag manifold endowed an invariant metric $g$ and an almost complex structute $J$. Computing 
$g(JX,JY)$ for the Weyl basis chosen it is easy to see that $g$ is
\emph{almost Hermitian} with respect to $J$, i.e,
$g(JX,JY)=g(X,Y)$. The pair $(g,J)$ is called an \emph{invariant almost Hermitian structure} on $\F$.

\section{K\"ahler form}

Consider an invariant almost Hermitian structure $(g,J)$ on a flag manifold  $\F$. We denote by $\Omega=\Omega_{J,g}$ the correspondent K\"{a}hler form:
\begin{equation}
\Omega\left(X,Y\right)=g\left(X,JY\right), \hspace{1cm}X,Y\in\mathfrak{m}.\label{forma de
kahler}
\end{equation}

We also denote by $\Omega$, its natural extension to a $G$-invariant 2-form on  $\mathfrak{m}^{\mathbb{C}}$. On the Weyl basis $\Omega$ is given by
\[
\Omega(X_{\alpha},X_{\beta})=g\left(
X_{\alpha},JX_{\beta}\right)=-i\lambda_{k(\alpha)}\varepsilon_{k(\beta)}B\left(X_{\alpha},X_{\beta}\right)=
\left\{
\begin{array}
[c]{c}%
-i
\varepsilon_{k(\alpha)}\lambda_{k(\alpha)}, \quad \text{
if 
}\quad \beta=-\alpha, \\ \\
0, \quad \text{ otherwise},
\end{array}
\right.
\]
where $k(\alpha)$, $k(\beta)$ are the correspondents t-roots
to $\alpha$ and $\beta$ respectively. Then from decomposition
\begin{equation*}
\mathfrak{m}^{\mathbb{C}}=\mathfrak{m}^{\mathbb{C}}_{1}\oplus\mathfrak{m}^{\mathbb{C}}_{2}\oplus\cdots\oplus\mathfrak{m}^{\mathbb{C}}
_{2s}
\end{equation*}
in $\mathfrak{k}^{\mathbb{C}}$-irreducible and inequivalent submodules we obtain 
\[
\Omega\left(\cdot,\cdot\right)=\sum_{\alpha\in R_{M}^+}-i\varepsilon_{\alpha}\lambda_{\alpha}
B\left(\cdot,\cdot\right)|_{\mathfrak{g}^{\mathbb{C}}_{\alpha}\times\mathfrak{g}^{\mathbb{C}}_{-\alpha}}
=\sum_{\delta\in R_{\mathfrak{t}}^+}-i\varepsilon_{\delta}\lambda_{\delta}\left(\sum\limits_{\substack{\alpha\in R_{M}^+\\k(\alpha)=\delta}}
B\left(\cdot,\cdot\right)|_{\mathfrak{g}^{\mathbb{C}}_{\alpha}\times\mathfrak{g}^{\mathbb{C}}_{-\alpha}}\right)
=\sum_{\delta\in R^{+}_{\mathfrak{t}}}
-i\varepsilon_{\delta}\lambda_{\delta}
B\left(\cdot,\cdot\right)|_{\mathfrak{m}^{\mathbb{C}}_{\delta}\times\mathfrak{m}^{\mathbb{C}}_{-\delta}}.
\]

The exterior differential $d\Omega$ computed at $T_o\F$ is
$$
3d\Omega(X,Y,Z)=-\Omega([X,Y],Z)+\Omega([X,Z],Y)-\Omega([Y,Z],X)
$$
for $X,Y,Z\in\mathfrak{m}^{\mathbb{C}}$ (cf. \cite{Kob}, 3.11). An easy computation on Weyl basis obtained in \cite{SM N} shows that if $\alpha,\beta,\gamma\in R_{M}$  then 
\begin{equation*}
d\Omega(X_{\alpha},X_{\beta},X_{\gamma})=\left\{
\begin{array}
[c]{c}%
-3in_{\alpha,\beta}\left(\varepsilon_{\alpha}\lambda_{\alpha}
+\varepsilon_{\beta}\lambda_{\beta}+\varepsilon_{\gamma}\lambda_{\gamma}\right),\quad\text{
if 
}\quad \alpha+\beta+\gamma=0, \\ \\
0,\quad \text{ otherwise.}
\end{array}
\right.
\end{equation*}

The exterior differential $d\Omega$ has a similar look in terms of t-roots. To prove it we need the following result.

\begin{lemma}(\cite{Alek-Arv}, Lemma 4)\label{lema soma} Let $\xi,\eta,\zeta$ be t-roots such that $\xi+\eta+\zeta=0$. Then there exist roots 
$\alpha,\beta,\gamma\in R_{M}$ with
$k(\alpha)=\xi,k(\beta)=\eta,k(\gamma)=\zeta$, satisfying 
$\alpha+\beta+\gamma=0$.
\end{lemma}
%
%

\begin{proposition} \label{12 simp}
Let $\delta,\zeta,\eta$ be t-roots then
$d\Omega(\mathfrak{m}^{\mathbb{C}}_{\delta},\mathfrak{m}^{\mathbb{C}}_{\zeta},\mathfrak{m}^{\mathbb{C}}_{\eta})=\left\{0\right\}$,
unless if  $\delta,\zeta,\eta$ are tzs on $R_t$. In this case
\[
d\Omega(X,Y,Z)=-3ir\left(\varepsilon_{\delta}\lambda_{\delta}
+\varepsilon_{\zeta}\lambda_{\zeta}+\varepsilon_{\eta}\lambda_{\eta}\right).
\]
where $r$ is a non zero constant and $X,Y,Z$ are vectors in
$\mathfrak{m}^{\mathbb{C}}_{\delta},\mathfrak{m}^{\mathbb{C}}_{\zeta}$
and $\mathfrak{m}^{\mathbb{C}}_{\eta}$, respectively.
\end{proposition}
\begin{proof}

If $\delta,\zeta,\eta$ are not tzs on $R_t$ it is easy to see that $d\Omega(\mathfrak{m}^{\mathbb{C}}_{\delta},\mathfrak{m}^{\mathbb{C}}_{\zeta},\mathfrak{m}^{\mathbb{C}}_{\eta})=\left\{0\right\}$. 

On the other hand, if $\delta,\zeta,\eta$ are tzs on $R_t$ let  $\alpha,\beta,\gamma$ be roots in $R_{M}$ with
$k(\alpha)=\delta$, $k(\beta)=\zeta$ and $k(\gamma)=\eta$, satisfying
$\alpha+\beta+\gamma=0$, according to the previous Lemma.
Thus, from the description of invariant metrics and iacs, we obtain:
\[
d\Omega(X_{\alpha},X_{\beta},X_{\gamma})=-3in_{\alpha,\beta}\left(\varepsilon_{\alpha}\lambda_{\alpha}
+\varepsilon_{\beta}\lambda_{\beta}+\varepsilon_{\gamma}\lambda_{\gamma}\right)
=-3in_{\alpha,\beta}\left(\varepsilon_{\delta}\lambda_{\delta}
+\varepsilon_{\zeta}\lambda_{\zeta}+\varepsilon_{\eta}\lambda_{\eta}\right).
\]
Now note that if  $\alpha',\beta',\gamma'\in R_{M}$ are such that
$k(\alpha')=\delta$, $k(\beta')=\zeta$, $k(\gamma')=\eta$ and
$\alpha'+\beta'+\gamma'=0$ then
\[
d\Omega(X_{\alpha'},X_{\beta'},X_{\gamma'})=
-3in_{\alpha',\beta'}\left(\varepsilon_{\delta}\lambda_{\delta}
+\varepsilon_{\zeta}\lambda_{\zeta}+\varepsilon_{\eta}\lambda_{\eta}\right).
\]
Now let $X$,$Y$ and $Z$ be non zero vector in $\mathfrak{m}^{\mathbb{C}}_{\delta},\mathfrak{m}^{\mathbb{C}}_{\zeta}$
and $\mathfrak{m}^{\mathbb{C}}_{\eta}$, respectively. Using the Weyl basis we can write  $X=\sum_{\alpha}a_\alpha X_\alpha$, $Y=\sum_{\beta}b_\beta X_\beta$, $Z=\sum_{\gamma}c_\gamma X_\gamma$, where the sums is taken over all roots $\alpha$, $\beta$ and $\gamma$ in $R_M$ such that $k(\alpha)=\delta$, $k(\beta)=\zeta$ and $k(\gamma)=\eta$ satisfying
$\alpha+\beta+\gamma=0$. Then, using the linearity of the exterior differential 
\begin{eqnarray}
d\Omega(X,Y,Z)&=&
-3i\sum_{\alpha, \beta, \gamma}a_\alpha b_\beta c_\gamma n_{\alpha,\beta}\left(\varepsilon_{\alpha}\lambda_{\alpha}
+\varepsilon_{\beta}\lambda_{\beta}+\varepsilon_{\gamma}\lambda_{\gamma}\right)\nonumber\\
&=& -3i\left(\varepsilon_{k(\alpha)}\lambda_{k(\alpha)}
+\varepsilon_{k(\beta)}\lambda_{k(\beta)}+\varepsilon_{k(\gamma)}\lambda_{k(\gamma)}\right)\sum_{\alpha, \beta, \gamma}a_\alpha b_\beta c_\gamma n_{\alpha,\beta}\nonumber\\
&=& -3ir\left(\varepsilon_{\delta}\lambda_{\delta}
+\varepsilon_{\zeta}\lambda_{\zeta}+\varepsilon_{\eta}\lambda_{\eta}\right)\nonumber\\
\end{eqnarray}
where $r=\sum_{\alpha, \beta, \gamma}a_\alpha b_\beta c_\gamma n_{\alpha,\beta}$. Note that $r$ is not zero because $\alpha+\beta$ is a root.

\end{proof}

\begin{definition}
Let 
$J=\left\{\varepsilon_{\delta},\delta\in R_t\right\}$
be an iacs on $\F$. Let  $\{\delta,\zeta,\eta\}$ be a tzs of t-roots. The triple $\{\delta,\zeta,\eta\}$ is said to be 
\begin{enumerate}
\item a $(0,3)$-triple of t-roots if
$\varepsilon_{\delta}=\varepsilon_{\zeta}=\varepsilon_{\eta}$;
\item a $(1,2)$-triple of t-roots, otherwise.
\end{enumerate}
\end{definition}

\begin{remark}
The notion of  $(0,3)$-triple and $(1,2)$-triple of roots in $R$ was introduced in \cite{SM N} for full flag manifolds and in \cite{SM R} for partial flag manifolds. It is easy see that a tzs 
$\alpha,\beta,\gamma$ in $R_M$ is a (0,3)-triple (resp. (1,2)-triple) iff  $k(\alpha),k(\beta),k(\gamma)$ is a (0,3)-triple of t-roots
 (resp. (1,2)-triple of t-roots). In this sense, the above definition extends that one in \cite{SM N} or \cite{SM R}.
\end{remark}

\begin{lemma} \label{prop ics}
Let $\F$ endowed with the complex structure $J$ associated to a fixed choice of positive roots $R_M^+$ in $R_M$, then $J$ has no (0,3)-triple of t-roots. Reciprocally if an iacs $J$ has no (0,3)-triple of t-roots then $J$ is a complex structure on $\F$.
\end{lemma}

\begin{proof}
Let $\{\delta, \zeta, \eta\}$ be a tzs of t-roots and $\{\alpha, \beta, \gamma\}$ a correspondent tzs in $R$, such that $k(\alpha)=\delta$, $k(\beta)=\zeta$ and $k(\gamma)=\eta$. Two of the roots $\{\alpha, \beta, \gamma\}$ are positive and the other is negative or two is negative and the other is positive, with respect to the choice of positive roots $R_M^+$ associated to $J$. Then, from (\ref{ics}), the triple  $\{\varepsilon_{\alpha},\varepsilon_{\beta},\varepsilon_{\gamma}\}$ has different signals. From Proposition \ref{iacs}, $\varepsilon_{\alpha}=\varepsilon_{\delta}$, $\varepsilon_{\beta}=\varepsilon_{\zeta}$, $\varepsilon_{\gamma}=\varepsilon_{\eta}$, then $\{\delta, \zeta, \eta\}$ is a (1,2)-triple of t-roots.

On the other hand, if $J$  has no (0,3)-triple of t-roots then from Proposition \ref{J complex} it is clear that $J$ is a complex structure. 
\end{proof}

An almost Hermitian manifold 
$(M,g,J)$ is \emph{(1,2)-sympletic (or quasi K\"{a}hler)} if
\[
d\Omega\left(X,Y,Z\right)=0
\]
when one of the vectors $X,Y,Z$ is of type $(1,0)$ and the other two of type $(0,1)$.

The next result classify the quasi K\"{a}hler class on flag manifolds in terms of t-roots.

\begin{proposition}\label{crit 12 simp}
An invariant almost Hermitian structure $\left(J=\left\{\varepsilon_{\delta}\right\},g=\left\{\lambda_{\delta}\right\}\right)$,
$\delta\in R_t$, on $\F$ is $(1,2)$-sympletic if and only if

\begin{equation} \label{(1,2) criterio}
\varepsilon_{\delta}\lambda_{\delta}+\varepsilon_{\zeta}\lambda_{\zeta}+\varepsilon_{\eta}\lambda_{\eta}=0
\end{equation}
for every $(1,2)$-triple of t-roots $\{\delta,\zeta,\eta\}$.
\end{proposition}
\begin{proof}
Consider a $(1,2)$-triple of t-roots $\{\delta,\zeta,\eta\}$ and let $\alpha$, $\beta$ and $\gamma$ in $R_M$ such that $\alpha+\beta+\gamma=0$ and $k(\alpha)=\delta$, $k(\beta)=\zeta$ and $k(\gamma)=\eta$. Then 
\begin{equation}\label{without 03 triples}
0=d\Omega(X_{\alpha},X_{\beta},X_{\gamma})=-3in_{\alpha,\beta}\left(\varepsilon_{\alpha}\lambda_{\alpha}
+\varepsilon_{\beta}\lambda_{\beta}+\varepsilon_{\gamma}\lambda_{\gamma}\right)
=-3in_{\alpha,\beta}\left(\varepsilon_{\delta}\lambda_{\delta}
+\varepsilon_{\zeta}\lambda_{\zeta}+\varepsilon_{\eta}\lambda_{\eta}\right).
\end{equation}
Since $n_{\alpha,\beta}\neq 0$ if $\alpha+\beta+\gamma=0$, we get (\ref{(1,2) criterio}). 
\end{proof}

An almost Hermitian manifold is said to be  \emph{almost K\"{a}hler}
if $\Omega$ is sympletic, i.e., $d\Omega=0$. When $d\Omega=0$
and $J$ is integrable the manifold is said to be  \emph{ K\"{a}hler},
(cf. \cite{Kob}). In the next result we used the idea performed in \cite{SM N} and \cite{SM R} for t-roots.
\begin{proposition}\label{12 simp e kahler}
An invariant almost Hermitian structure on $\F$ is almost 
K\"{a}hler if and only if it is K\"{a}hler.
\end{proposition}
\begin{proof}
Note that the a pair $(g,J)$ almost K\"{a}hler do not admits $(0,3)$-triples of t-roots. Because in this case, from (\ref{without 03 triples}), we would have the following equation
\[
\lambda_{\delta}+\lambda_{\zeta}+\lambda_{\eta}=0
\]
that is impossible since 
$\lambda_{\delta},\lambda_{\zeta},\lambda_{\eta}>0$. Thus the iacs 
$J$ admits only $(1,2)$-triples of t-roots and in this case, from Lemma \ref{prop ics}, $J$ is integrable and the pair $(g,J)$ is K\"{a}hler.

Of course if the pair $(g,J)$ is K\"{a}hler then it is almost K\"{a}hler.
\end{proof}

\section{$\mathcal{G}_1$  structures}

An invariant almost Hermitian structure $(g,J)$ on $\F$ is called a $\mathcal{G}_1$  structure (or $W_1\oplus W_3\oplus W_4$), according the sixteen classes in \cite{GH}, if 
$$
g(N(X,Y),X)=0,\quad X,Y\in \mathfrak{m}.
$$

It is clear that if  $J$ is a complex structure, i.e. $N(X,Y)=0$, then $(g,J)$ is a $\mathcal{G}_1$ structure for any invariant metric $g$ on $\F$. 



\begin{theorem}\label{J for all}
Fixed a choice of positive roots $R_M^+$ in $R_M$, the correspondent complex structure $J$ on $\F$ is a $\mathcal{G}_1$ structure with respect to all invariant metric $g$ on $\F$. 
\end{theorem}

\begin{corollary}\label{G1 iso irr}
If $\F$ is isotropy irreducible, then any invariant pair $(g,J)$ is a $\mathcal{G}_1$ structure.
\end{corollary}
\begin{proof}
Follows from Corollary \ref{J iso irr} and the previous Theorem.
\end{proof}
Now we study the annihilation of the tensor $g(N(X,Y),X)$ taking in account the vector roots from the Weyl basis. From (\ref{nij}), if $\alpha, \beta, \gamma$ are roots in $R_M$ it easy to see that 

\begin{equation} \label{g1 s}
-\dfrac{1}{2}g(N(X_{\alpha},X_{\beta}),X_{\gamma})=\left\{
\begin{array}
[c]{c}%
\lambda_\gamma n_{\alpha,\beta}(\varepsilon_\alpha\varepsilon_\beta +\varepsilon_\alpha\varepsilon_\gamma+\varepsilon_\beta\varepsilon_\gamma+1),
\quad\text{
if 
}\quad \alpha+\beta+\gamma=0, \\ \\
0, \hspace{0.5cm} \text{ otherwise},
\end{array}
\right.
\end{equation}
in particular $g(N(X_\alpha, X_\beta), X_\alpha)=0$, for $\alpha, \beta \in R_M$. 

Using similar arguments of \cite{SM N}, in the next result we obtain a classification of $\mathcal{G}_1$  structures in terms of t-roots. 

\begin{lemma}\label{g1 criterio}
The pair $(g,J)$ is a $\mathcal{G}_1$  structures on $\F$ if and only if $\lambda_{\delta}=\lambda_{\zeta}=\lambda_{\eta}$ if $\{\delta,\zeta,\eta\}$ is a  (0,3)-triple of t-roots.
\end{lemma}
\begin{proof}
Let $\{\delta,\zeta,\eta\}$ is a tzs of t-roots and $\alpha, \beta, \gamma \in R_M$ a tzs of roots such that $k(\alpha)=\delta$, $k(\beta)=\zeta$ and $k(\gamma)=\eta$.
Note that
\begin{eqnarray}
-\dfrac{1}{2}\left\lbrace g(N(X_{\alpha},X_{\beta}),X_{\gamma})+g(N(X_{\gamma},X_{\beta}),X_{\alpha})\right\rbrace
&=& n_{\alpha,\beta} (\lambda_\gamma-\lambda_\alpha) (\varepsilon_\alpha\varepsilon_\beta +\varepsilon_\alpha\varepsilon_\gamma+\varepsilon_\beta\varepsilon_\gamma+1)\nonumber\\ \nonumber\\
&=& n_{\alpha,\beta} (\lambda_\eta-\lambda_\delta) (\varepsilon_\delta\varepsilon_\zeta +\varepsilon_\delta\varepsilon_\eta+\varepsilon_\zeta\varepsilon_\eta+1)\label{g1}
\end{eqnarray}
where in first equality we used that $ n_{\alpha,\beta}=-n_{\gamma,\beta}$ if $\alpha+\beta+\gamma=0$. Now we note that $(\varepsilon_\delta\varepsilon_\zeta +\varepsilon_\delta\varepsilon_\eta+\varepsilon_\zeta\varepsilon_\eta+1)$ is not zero if and only if $\{\delta,\zeta,\eta\}$ is a (0,3)-triple of t-roots.  Thus if $(g,J)$ is a $\mathcal{G}_1$  structure then $\lambda_\eta=\lambda_\delta$. Analogously one proves that $\lambda_\zeta=\lambda_\delta$. Reciprocally, for $X=\sum_\alpha a_\alpha X_\alpha$ with $\alpha\in R_M$, we obtain:
\[
g(N(X,X_\beta),X)=\sum_{\alpha\neq\gamma}a_\alpha a_\gamma\left\lbrace g(N(X_{\alpha},X_{\beta}),X_{\gamma})+g(N(X_{\gamma},X_{\beta}),X_{\alpha})\right\rbrace.
\]
Then, from (\ref{g1}),  $g(N(X,X_\beta),X)=0$  if $\lambda_{\delta}=\lambda_{\zeta}=\lambda_{\eta}$ when  $\{\delta,\zeta,\eta\}$ is a  (0,3)-triple of t-roots.
\end{proof}

As we saw, given a complex structure $J$ on $\F$, for any invariant metric $g$, the pair $(g,J)$ is a $\mathcal{G}_1$  structure. Thus fixed any invariant metric $g$ on $\F$, the $\mathcal{G}_1$  structures $(g,J)$ are in correspondence one-to-one with: parabolic subalgebras $\mathfrak{p}_\Theta$ in $\mathfrak{g}$; basis $\Sigma$ of the root system $R$ which contain a fixed basis $\Theta$ of the root system $R_{\Theta}$; basis of system $R_t$; t-chambers; a choice of positive roots $R^+_M$ in $R_M$ and reflections $\omega$ of the Weyl group (of the root system $R$) which satisfy $\omega\Theta\subset\Sigma$, according to Section 5.

 Let $J=\{\varepsilon_\delta, \delta\in R_t\}$ be an iacs on $\F$ and denote by $C(J)$ the subset of t-roots $\delta$ such that there exist $(0,3)$-triple of t-roots $\{\delta,\zeta,\eta\}$ containing $\delta$.

\begin{proposition}
Given an iacs $J$ on $\F$, a pair $(g,J)$ is a $\mathcal{G}_1$ struture if and only if $g=\{\lambda_\delta\}$ is constant on $C(J)$.
\end{proposition}
\begin{proof}
Follows immediately from Lemma \ref{g1 criterio}.
\end{proof}

Analogously, let $g=\{\lambda_\delta, \delta\in R_t\}$ be an invariant metric on $\F$ and denote by $C(g)$ the subset of t-roots $\delta$ such that there exist tzs $\{\delta,\zeta,\eta\}$ of t-roots containing $\delta$ and satisfying $\lambda_{\delta}=\lambda_{\zeta}=\lambda_{\eta}$.

\begin{proposition}
A pair $(g,J)$ is a $\mathcal{G}_1$ structure on $\F$ if and only if $C(J)\subset C(g)$.
\end{proposition}
\begin{proof}
Suppose that $(g,J)$ is a $\mathcal{G}_1$ structure and $C(g)=\emptyset$. Then $J$ has no $(0,3)$-triple of t-roots, from Lemma \ref{g1 criterio}. Thus $C(J)=\emptyset$ or $J$ is a invariant complex structure.

Now let $(g,J)$ be a $\mathcal{G}_1$ structure such that $C(J)\neq \emptyset$ and let $\delta$ be a t-root in $C(J)$. Then there exist a $(0,3)$-triple $\{\delta,\zeta,\eta\}$ of t-roots containing $\delta$. From Lemma \ref{g1 criterio}, the metric $g$ satisfies $\lambda_{\delta}=\lambda_{\zeta}=\lambda_{\eta}$, so $\delta$ is a t-root in $C(g)$ and $C(J)\subset C(g)$. 

Finally, if $C(J)\subset C(g)$ follows imediately from Lemma \ref{g1 criterio} that $(g,J)$ is a $\mathcal{G}_1$ structure on $\F$.
\end{proof}

As we saw in Theorem \ref{J for all}, any invariant complex structure $J$ provides a $\mathcal{G}_1$ structure with respect to all invariant metric on $\F$. Thus a natural question is: Is there an invariant metric $g_0$ on $\F$ such that $(g_0,J)$ is a $\mathcal{G}_1$ structure for any iacs on $\F$? The next result shows that, up to homotheties, there exists a unique invariant metric $g$ which is $\mathcal{G}_1$  structure with respect to every iacs $J$ on $\F$.

\begin{theorem} \label{metric G1}
The normal metric is the unique invariant metric which is a $\mathcal{G}_1$  structure on $\F$ with respect to any iacs $J$.
\end{theorem}

\begin{proof}
If $\F$ is isotropy irreducible the result follows from Corollary \ref{G1 iso irr}. 
Let $|R_t^+|>1$, from Lemma \ref{g1 criterio} it is easy to see that the normal metric is $\mathcal{G}_1$  structure on $\F$ with respect to any iacs $J$.
Reciprocally, suppose an invariant metric $g$ is $\mathcal{G}_1$  structure on $\F$ with respect to any iacs $J$. Consider $\delta$ and $\eta$, $\eta\neq \pm \delta$, two t-roots and $\lambda_\delta$, $\lambda_\eta$ the parameters of the invariant metric $g$ associated to these two t-roots, respectively.  By Theorem \ref{connected tzs t-roots}, there exists a chain $T^1,\dots, T^n $ of tzs of t-roots, connecting $\delta$ and $\eta$, such that $\delta\in T^1$ and $\eta\in T^n$. For each tzs $T_{i,j,k}^p=\{\xi_i,\xi_j,\xi_k\}\subset R_t$, we can take an iacs $J_p$ such that $\{\xi_i,\xi_j,\xi_k\}$ is a (0,3)-triple of roots, then $\lambda_{\xi_i}=\lambda_{\xi_j}=\lambda_{\xi_k}$, since $g$ is  $\mathcal{G}_1$  structure on $\F$ with respect to any iacs $J_p$. Continuing with this finite process we get that $\lambda_\delta=\lambda_\eta$.
\end{proof}

\section{Equivalent structures}

Consider the Weyl group $\mathcal{W}$ of the root system $R$. Recall that if $\omega\in \mathcal{W}$, then $\omega^2=1$ and $\omega^{-1}=\omega$. Let $A_\Theta$ be the subgroup of $\mathcal{W}$ of the reflection which preserves $R_\Theta \subset R$. Let $\omega\in\mathcal{W}$ be a reflection, it is not difficult to see that  $\omega R_M\subset R_M$ if and only if $\omega \in A_\Theta$. Thus a reflection $\omega$ of $A_\Theta$ permites the roots $\beta$ in $R_M$ and the related eigenspaces $\mathfrak{g}_\beta$. If $\beta\in R_M$, we denote by $k(\beta)$ the correspondent t-root. The group $A_\Theta$ acts on the set of iacs by 
$$
\omega\cdot J=\omega\cdot \{\varepsilon_{k(\beta)}\}=\{\varepsilon_{k(\omega\beta)}\}, \quad \beta \in R_M,\quad \omega\in A_\Theta.
$$
This action is well defined because if $\alpha, \beta \in R_M$ are such that $k(\alpha)=k(\beta)$ and $\omega \in A_\Theta$ then $k(\omega\alpha)=k(\omega\beta)$. Indeed, let $\Sigma=\{\alpha_1, \dots \alpha_k,\beta_1, \dots, \beta_r \}$ be a basis of $R$ associated to $\F$ where $\Theta=\{\alpha_1, \dots \alpha_k\}$ is a basis of $R_\Theta$ and $\Sigma_M=\Sigma\setminus\Theta=\{\beta_1, \dots, \beta_r \}$. Then  $\alpha=\sum p_i\alpha_i+\sum m_j \beta_j$,  $\beta=\sum q_i\alpha_i+\sum m_j \beta_j$ and $k(\omega \alpha)=k(\omega \beta)$, since $\omega \alpha_i\in R_\Theta$ for all $\omega\in A_\Theta$. It is known that the iacs $J$ and $\omega\cdot J$ are associated by a bi-holomorphic map on $\F$. In this sense, if $\omega\in A_\Theta$ we say that the iacs $J$ and $\omega\cdot J$ are equivalents. In particular,  $J_c$ and $\omega\cdot J_c$ are equivalents structures, (see  \cite{Alek e Perol}).

In a similiar way, the $A_\Theta$ act on the set of invariant metrics on $\F$ by
$$
\omega\cdot g=\omega\cdot \{\lambda_{k(\beta)}\}=\{\lambda_{k(\omega\beta)}\},\quad \beta \in R_M,\quad \omega\in A_\Theta.
$$

Following \cite{SM N}, if $\omega\in A_{\Theta}$ we say that the invariant almost Hermitian structures $(g, J)$ and $(\omega\cdot g, \omega\cdot J)$ are equivalent. The structures $(g, J)$ and $(\omega\cdot g, \omega\cdot J)$ share the same class of invariant almost Hermitian structures.  


\end{document}